\documentclass[10pt,A4paper]{amsart}

\usepackage{extensions}

\begin{document}
\title[]{Poincaré inequalities and integrated curvature-dimension criterion for generalised Cauchy and convex measures}
%\date{\today\ \emph{ File: }\jobname.tex}

\author[B. Huguet]{Baptiste Huguet} 
\address{IRMAR (ENS Rennes), UMR CNRS 6625, Univ. Rennes, France}
\email{baptiste.huguet@math.cnrs.fr}
\urladdr{https://www.math.cnrs.fr/~bhuguet/}

\begin{abstract}
We obtain new sharp weighted Poincaré inequalities on Riemannian manifolds for a general class of measures. When specialised to generalised Cauchy measures, this gives a unified  and simple proof of the weighted Poincaré inequality for the whole range of parameters, with the optimal spectral gap, the error term and the extremal functions.
\end{abstract}
%%%%%%%%%%%%%%%%%%%%%%%%%%%%%%%%%%%%%%%%%%%%%%%%%%%%%%%%%%%%%%%%%%%%%%%%%%%%
%%%%%%%%%%%%%%%%%%%%%%%%%%%%%%%%%%%%%%%%%%%%%%%%%%%%%%%%%%%%%%%%%%%%%%%%%%%%
\maketitle
\setcounter{tocdepth}{1}
\tableofcontents

%%%%%%%%%%%%%%%%%%%%%%%%%%%%%%%%%%%%%%%%%%%%%%
%%%% Main text entry area:
\section{Introduction}
The family of generalised Cauchy distributions is a family of probability measures defined on $\R^n$ by~$d\mu_\beta\sim (1+|x|^2)^{-\beta} dx$, for $\beta>n/2$. This family plays a significant role in a variety of mathematics issues. For example, it acts as the heat kernel for fast diffusions. In this context, it is known as the Barenblatt profile. \cite{BobLed} pointed out that generalised Cauchy measures approach the Gaussian measure when $\beta$ diverges, after rescaling. This prompted them to investigate the properties of these measures in terms of functional inequalities, specifically Poincaré inequalities. Due to heavy tails, these measures cannot satisfy the classical Poincaré inequality. Nevertheless, a weighted version of it has been proved, with weight~$\omega(x) = 1+|x|^2$. For all $n\geq1$ and all $\beta\geq n$ (or~$\beta>1$ if~$n=1$), we have
\[\frac{2(\beta-1)}{\left(\sqrt{1+\frac{2}{\beta-1}}+\sqrt{\frac{2}{\beta-1}}\right)^2}\Var_{\mu_\beta}(f)\leq \int_{\R^n}|df|^2(1+|x|^2)\, d\mu_\beta,\]
for all smooth bounded functions $f$. Their result is sufficiently sharp to recover the classical Poincaré inequality for Gaussian measures in $\R^n$, by rescaling. However, it does not cover the entire range $n/2<\beta$ and the constant is not optimal. 

Poincaré inequalities, whether classical or weighted, are equivalent to spectral properties of a suitable operator. In the case of generalised Cauchy measures, this operator is $Lf = \omega\Delta f -2(\beta-1)\langle df, x\rangle$, defined on smooth compactly supported functions $f\in\CC^\infty_c(\R^n)$ and extended to its domain $\DD(L)\in\L^2(\mu_\beta)$. The spectrum of this operator was studied by \cite{DMc} from the perspective of fast diffusion equations, providing a thorough description of its bottom, its eigenvalues and associated eigenfunctions. Their proof is based on a decomposition along spherical harmonics. This description allows \cite{BBDGV} to prove weighted Poincaré inequalities with optimal constant. However, this approach is highly specific to generalised Cauchy measures.  

These articles paved the way to many projects on weighted Poincaré inequalities. On the one hand, these works tackle the question towards more general measures. The  article \cite{BobLed} has studied the general broad framework of convex measures. These measures are defined using a Brunn-Minkowski-type inequality. In $\R^n$, they can be characterised as probability measures whose density is proportional to $\omega^{-\beta}$ for a positive convex function $\omega$ defined on some convex set and $\beta\geq n$. This is the context in which most of the work has been carried out. Their aim is to establish weighted Poincaré inequalities with sharp control over the optimal constant.  On the other hand, they question the relevance of probabilistic arguments to spectral theory, in particular, the relevance of Markov semigroups and curvature-dimension theory. These works often lead to more general functional inequalities, such as logarithmic-Sobolev or other~$\Phi$-entropies. 

Sharp weighted Poincaré inequalities were proved by \cite{Ngu}, for convex measures, with an additional assumption of strong convexity, that is~$d\mu \sim \omega^{-\beta} dx$ on $\R^n$, where $\omega$ is smooth and $\Hess(\omega)$ uniformly lower bounded. His proof is inspired by Hörmander's $L^2$-method. His result depends on $\Hess(\omega)$ explicitly, but is only valid for $\beta\geq n+1$. In the case of generalised Cauchy measures, his result is optimal and associated with an eigenfunction. 

The case of generalised Cauchy measures was completed by \cite{BJM2} in dimension one using the intertwining approach. A milestone was reached by \cite{BJM} for the $n$-dimensional case. They obtained a general spectral gap comparison between a radial measure and its one-dimensional radial part. When applied to generalised Cauchy measures, they prove weighted Poincaré inequalities for the whole range~${\beta>n/2}$. Nonetheless, they cannot reach the optimal constant over a small range of $\beta$. 

In addition to these major articles, several papers recovered partial results and enlightened new perspectives. Among them, \cite{Saum} used Stein kernels to recover the optimal bound in dimension $1$ for $\beta >3/2$. In \cite{ABJ}, a new bound was obtained for $\beta>n$, using metric twisting and intertwining. Unfortunately, their bound is not optimal. In \cite{Hug}, the previous method was generalised to Riemannian manifolds. Interpreting the weight as a metric, it can be applied to generalised Cauchy measures in dimension~$n=2$. Then it provides a better bound, even optimal for $1<\beta\leq 1+\sqrt{2}$. The optimal bound was recovered for $\beta\geq n+1$ in a Riemannian setting by \cite{BGS}, using harmonicness and curvature-dimension arguments. Their result is a geometric generalisation of \cite{Ngu}. This Riemannian result was recovered by \cite{GZ}, using the curvature-dimension criterion on the Laplace operator only, with applications to Poincaré and Beckner inequalities.

Each one of these articles uses different perspectives and different techniques. Together, they do not allow to properly understand the link between curvature-dimension criterion and Poincaré inequality, especially for generalised Cauchy measures. The goal of the present work is to obtain a simple, self-contained and intrinsic proof of sharp Poincaré inequalities using curvature-dimension arguments only. A novelty in our work is to find out a result beyond the condition $\beta\geq n+1$ and even beyond convex measure condition $\beta\geq n$. When specialised to generalised Cauchy measures, we recover the spectral gap $\lambda_1(-L)$, that is to say the optimal constant such that
\[\lambda_1(-L)\Var_{\mu_\beta}(f)\leq\int_{\R^n}|df|^2(1+|x|^2)\, d\mu_\beta,\]
for the whole range of parameters. Our guiding principle is to recover the following result from \cite{DMc}. 

\begin{theorem}\label{prop:spec}
For $n=1$, we have
\[\lambda_1(-L) =\left\{
\begin{array}{ll} 
(\beta-1/2)^2 & \text{if}\quad 1/2<\beta\leq 3/2\\
2(\beta-1) & \text{if}\quad 3/2\leq \beta.
\end{array}\right.\]
For $n\geq2$, we have
\[\lambda_1(-L) =\left\{
\begin{array}{ll} 
(\beta-n/2)^2 & \text{if}\quad n/2<\beta\leq n/2+2\\
4(\beta-n/2-1)& \text{if}\quad n/2+2\leq\beta\leq n+1\\ 
2(\beta-1) & \text{if}\quad n+1\leq \beta.
\end{array}\right.\]
\end{theorem}

By doing so, we obtain some general Poincaré inequalities for a large class of measures. Also, we obtain an estimation of the error term. In the case of generalised Cauchy measures, this estimation is exact and sufficiently explicit to characterise the associated extremal functions, when they exist. Our method is inspired by the integrated curvature-dimension criterion. This criterion has been firstly introduced by \cite{Oba}, for the Laplace operator on compact Riemannian manifolds. The generalisation of this criterion thank to the carré du champ formalism was introduced by \cite{BE}. and developed by \cite{Led5}. An improvement of this criterion was proposed by \cite{ChL}, with an explicit expression of the error term, or deficit, in Poincaré inequality. This improvement allows to characterise extremal functions.

Let us describe the structure of this article. In Section \ref{Sec:var} we present the integrated curvature-dimension criterion and its link to weighted Poincaré inequalities. In Section \ref{Sec:opt}, we obtain some upper bound of the spectral gap for generalised Cauchy measures by studying its eigenfunctions. In Section \ref{Sec:wPI}, we prove weighted Poincaré inequalities for the measures $d\mu \sim \omega^{-\beta}$, where $\omega$ is smooth, positive and strongly convex. These inequalities will be satisfied under assumptions which need to be discussed. We apply the results to generalised Cauchy measures so as to obtain the exact spectral gap and the extremal functions. Finally, in Section \ref{Sec:1d}, we return to the one-dimensional case, as a toy application of our method. This section can be read directly after Section \ref{Sec:var}, as it exposes most of our arguments in a simpler way.

%%%%%%%%%%%%%%%%%%%%%%%%%%%%%%%%%%%%%%%%%%%%%%%%%%%%%%%%%%%%%%%%%%%%%%%%%%%%
%%%%%%%%%%%%%%%%%%%%%%%%%%%%%%%%%%%%%%%%%%%%%%%%%%%%%%%%%%%%%%%%%%%%%%%%%%%%
\section{Variance representation formula}
\label{Sec:var}
Let $M$ be a complete Riemannian manifold of dimension $n$. Let $\mu$ be a probability measure on $M$ and $L$ a reversible diffusion operator associated with $\mu$. We denote by $\BP$ the $\L^2$-semigroup generated by $L$. We assume that the semigroup~$\BP$ is ergodic. Let $\Gamma$ and $\Gamma_2$ be the carré du champ and iterated carré du champ  operators, respectively defined on $f\in\CC^\infty_c(M)$ by
\[ 2\Gamma(f) = L(f^2)-2fLf\quad\text{and}\quad 2\Gamma_2(f) = L\Gamma(f)-2\Gamma(f,Lf).\]

Under suitable conditions, discussed in \cite{BGL}, these operators can be uniquely extended to the domain $\DD(L)\subset\L^2(\mu)$. The link between these operators and  functional inequalities satisfied by the measure $\mu$ has been studied since the work of \cite{BE}. The operator $L$ satisfies the curvature-dimension criterion $CD(\rho,\infty)$, also called Bakry-\'Emery criterion, with $\rho\in\R$, if
\begin{equation}\label{BE}
\Gamma_2\geq \rho\Gamma .
\end{equation}
Moreover, if $\rho>0$, then under the $CD(\rho,\infty)$ criterion, the associated measure $\mu$ satisfies a Poincaré-type inequality with constant $\rho$:
\begin{equation}\label{eq:PI}
\rho\Var_\mu(f)\leq \int_M\Gamma(f)\, d\mu.
\end{equation}
For the operator $L=\Delta - \nabla V\cdot\nabla$ associated with $d\mu = e^{-V} dx$, the carré du champ operator is $\Gamma(f) = |df|^2$ and \eqref{eq:PI} is the classical Poincaré inequality. However, for a general $L$, this inequality is a weighted Poincaré inequality. Actually, both Poincaré inequality and weighted Poincaré inequality are equivalent to a spectral gap property on the generator $L$. We denote by $\lambda_1(-L)$ this spectral gap: 
\[\lambda_1(-L) = \inf\left\{\frac{\int_M-fLf\, d\mu}{\Var_\mu(f)}, f\in\DD(-L)\right\}.\]

Actually, curvature-dimension criterion implies stronger functional inequalities than Poincaré-type inequalities, such as logarithmic-Sobolev inequalities. In \cite{Led5}, it was proved in $\R^n$ that the classical Poincaré inequality is equivalent to an integrated version of the Bakry-\'Emery criterion. This result is very general and has been adapted to a larger class of diffusion and space. See \cite{BGL} or \cite{CG} for more details. Here, we recall a  version of this result from \cite{ChL}, where the deficit term is characterised.  

\begin{prop}[\cite{ChL}]\label{prop:PI}
For all $f\in\CC^\infty_c(M)$ and for all $\rho\neq0$ we have
\[\Var_\mu(f) = \frac{1}{\rho}\int_M\Gamma(f)d\mu -\frac{2}{\rho}\int_0^{+\infty}\int_M(\Gamma_2-\rho\Gamma)(\BP_tf)\, d\mu\, dt.\]
In particular, $\lambda_1(-L)$ is the largest $\rho>0$ such that 
$\int_M (\Gamma_2-\rho\Gamma)\, d\mu \geq 0.$
\end{prop}

\begin{proof}
Let us define $h(t) = \int_M (\BP_tf)^2\, d\mu$. We have
\[\Var_\mu(f)= -\left[h_t\right]_0^{+\infty} = \frac{-1}{2\rho}h'(0) -\frac{1}{\rho}\int_0^{+\infty}\frac{1}{2}h''(t)+\rho h'(t)\, dt.\]

Then, we compute the first derivatives of $h$. 
\[h'(t)= \int_M \partial_t (\BP_tf)^2\, d\mu = \int_M 2\BP_tf L\BP_tf\, d\mu= -2\int_M \Gamma(\BP_tf)\, d\mu,\]

\[h''(t) = -2 \int_M \partial_t\Gamma(\BP_tf)\, d\mu= -4\int_M \Gamma(\BP_tf, L\BP_tf)\, d\mu= 4\int_M \Gamma_2(\BP_tf)\, d\mu.\]

This proves the formula. Besides, if $L$ satisfies an integrated Bakry-\'Emery criterion, with constant $\rho>0$, then it is clear that $\lambda_1(-L)\geq\rho$. Conversely, using Cauchy-Schwarz inequality, we have
\[\int_M\Gamma(f)\, d\mu = \int_M \left(f-\int_M f\, d\mu\right)(-L)f\, d\mu \leq \Var_\mu(f)^{1/2}\left(\int_M\Gamma_2(f)\, d\mu\right)^{1/2}.\]

Hence, we have
\[\lambda_1(-L)\leq \inf_{f\in\CC_c^\infty(M)}\frac{\int_M\Gamma_2(f)\, d\mu}{\int_M\Gamma(f)\, d\mu}.\]
\end{proof}

Let us remark that this formula can be extended to other $\Phi$-entropy inequalities, such as logarithmic-Sobolev inequality, using
\[h(t) = \int_M \Phi(\BP_tf)\, d\mu.\]

The error term does not have an explicit expression because of the dependence on $\BP f$. Yet, it is sufficiently explicit to provide a necessary condition for extremal function, up to regularity considerations, or at least, a heuristic to find them. To do so, we firstly need to extend the error formula to the domain $\DD(L)$. This can be done if $L$ is essentially self-adjoint and if it satisfies any $CD(\rho,\infty)$ criterion for $\rho\in\R$ (for instance, see \cite{BGL}). Then, we also need some argument to prove the convergence of $\BP_t f$ as $t$ converges to $0$. Under these considerations, if $L$ satisfies an integrated curvature-dimension criterion with $\rho>0$ and if $f\in\DD(L)$ is extremal for the associated Poincaré inequality, then we have 
\[\int_M(\Gamma_2-\rho\Gamma)(f)\, d\mu = 0.\] 
This characterisation was used in \cite{ChL} on the Ornstein-Uhlenbeck operator to recover the extremal functions in Poincaré and logarithmic-Sobolev inequalities for the Gaussian measures.

Afterwards, we use Proposition \ref{prop:PI} to determine weighted Poincaré inequalities. In the case of generalised Cauchy measures, the expression of the error term is sufficiently explicit to characterise the extremal functions or to prove that there do not exist any extremal functions. 

%%%%%%%%%%%%%%%%%%%%%%%%%%%%%%%%%%%%%%%%%%%%%%%%%%%%%%%%%%%%%%%%%%%%%%%%%%%%
%%%%%%%%%%%%%%%%%%%%%%%%%%%%%%%%%%%%%%%%%%%%%%%%%%%%%%%%%%%%%%%%%%%%%%%%%%%%
\section{Spectral gap upper bound}
\label{Sec:opt}
The goal of this section is to introduce generalised Cauchy measures and recall the spectral gap upper bounds for generalised Cauchy measures These upper bounds were already obtained in \cite{BJM} for all parameters. Nevertheless, the study of eigenfunctions associated with these bounds, or the lack of eigenfunctions, will be a very useful clue in Section \ref{Sec:wPI} so as to choose the appropriate integration by parts formulae. The study of the upper bound highlights three different ranges of the parameter $\beta$, on which the spectral gap has three different behaviours. 

We denote by $\omega$ the function defined on $\R^n$ by $\omega(x) = 1+|x|^2$. The generalised Cauchy measure $\mu_\beta$ is the probability measure defined on $\R^n$, for $\beta>n/2$, by 
\[d\mu_\beta =\frac{1}{Z_{n,\beta}} \omega^{-\beta} dx,\quad\text{with}\quad Z_{n,\beta}=\frac{\pi^{n/2}\Gamma(\beta-n/2)}{\Gamma(\beta)}.\] 

This probability measure is associated with the symmetric operator, defined on $\CC_c^{\infty}(\R^n)$ by
\[Lf = \omega \Delta f -2(\beta-1)\langle x, df\rangle.\]
The carré du champ operator associated with $L$ is then
\[\Gamma(f) = \omega|df|^2.\]

The operator $L$ is essentially self-adjoint and can be extended in a unique operator, still denoted by $L$ defined on a domain $\DD(L)\subset\L^2(\mu_\beta)$. The spectrum of this operator was completely described in \cite{DMc} with a decomposition on spherical harmonics. Here, we are interested in its spectral gap and in the associated eigenfunctions only.

\begin{prop}
Let be $n\geq1$ and $\beta> n/2+1$. For all $v\in\R^n$, the linear function $x\mapsto\langle v,x\rangle$ belongs to $\L^2(\mu_\beta)$ and is an eigenfunction of $-L$ associated with the eigenvalue $2(\beta-1)$. In particular, for all $\beta> n/2+1$, we have \[\lambda(-L) \leq 2(\beta-1).\]

For $\beta>n/2+2$, the function $x\mapsto|x|^2-\frac{n}{2\beta-n-2}$ belongs to $\L^2(\mu_\beta)$ and is an eigenfunction associated with the eigenvalue $4(\beta-\frac{n}{2}-1)$. In particular, for all~$\beta> n/2+2$, we have
\[\lambda(-L)\leq 4\left(\beta-\frac{n}{2}-1\right).\]
\end{prop} 

For the smallest $\beta$, the spectral gap is not an eigenvalue any longer. In \cite{BBDGV}, it is interpreted as the bottom of the continuous spectrum. We give a new proof of this optimality based on the one-dimensional argument from \cite{BJM2}, recalled in Proposition \ref{prop:1dopt}. The sequence of functions exhibited in their proof is peculiar to the dimension $n=1$ and needs some adaptation to the $n$-dimensional case. It leads to the following result:

\begin{prop}
For $n\geq2$, $n/2<\beta\geq n/2+2$ and $\varepsilon<(2\beta-n)/4$ the functions $f_\varepsilon : x\mapsto \omega^\varepsilon(x)$ belongs to $\DD(L)$ and we have
\[\lim_{\varepsilon\uparrow(2\beta-n)/4} \frac{\int_\R\Gamma(f_\varepsilon)\,d\mu_\beta}{\Var_{\mu_\beta}(f_\varepsilon)}=\left(\beta-\frac{n}{2}\right)^2.\]
In particular, for all $\beta >n/2$, we have
\[\lambda(-L)\leq\left(\beta-\frac{n}{2}\right)^2.\]
\end{prop}

\begin{proof}
Firstly, we remark that $f_\varepsilon\in\L^2(\mu_\beta)$ if and only if $\varepsilon<(2\beta-n)/4$ but at the limit we still have $f_{(2\beta-n)/4}\in\L^1(\mu_\beta)$. So, we have 
\[\frac{\Var_{\mu_\beta}(f_\varepsilon)}{\int_{\R^n}f_\varepsilon^2\, d\mu_\beta} = 1-\frac{\left(\int_{\R^n}f_\varepsilon\, d\mu_\beta\right)^2}{\int_{\R^n}f_\varepsilon^2\, d\mu_\beta}\underset{\varepsilon\to(2\beta-n)/4}{\longrightarrow}1.\]

Then, for all $\varepsilon<(2\beta-n)/4$ and for all $x\in\R^n$, we have 
\begin{align*}
Lf_\varepsilon (x)
&=\varepsilon\omega^{\varepsilon-1}(x)L\omega(x) +\varepsilon(\varepsilon-1)\omega^{\varepsilon-2}(x)\Gamma(\omega)(x)\\
&=\varepsilon\omega^{\varepsilon-1}(x)\left[2n\omega(x)-4(\beta-1)|x|^2\right]+\varepsilon(\varepsilon-1)\omega^{\varepsilon-2}(x)\left[4\omega(x)|x|^2\right]\\
&= \varepsilon\left[2n-4(\beta-\varepsilon)\right]f_\varepsilon(x) +4(\beta-\varepsilon).
\end{align*}

Thus,
\[
\frac{\int_{\R^n}-f_\varepsilon Lf_\varepsilon\, d\mu_\beta}{\int_{\R^n}f^2_\varepsilon\, d\mu_\beta} 
= \varepsilon\left[4(\beta-\varepsilon)-2n\right] -4(\beta-\varepsilon)\frac{\int_{\R^n}f_\varepsilon(x)\, d\mu_\beta}{\int_{\R^n}f^2_\varepsilon\, d\mu_\beta}.
\]

So, we have
\[\lim_{\varepsilon\uparrow(2\beta-n)/4} \frac{\int_\R\Gamma(f_\varepsilon)\,d\mu_\beta}{\Var_{\mu_\beta}(f_\varepsilon)} =\lim_{\varepsilon\uparrow(2\beta-n)/4} \frac{\int_\R-f_\varepsilon Lf_\varepsilon\,d\mu_\beta}{\int_{\R^n}f^2_\varepsilon\, d\mu_\beta} = \left(\beta-\frac{n}{2}\right)^2.\]
\end{proof} 

Hence, we have obtained three different upper bounds for the spectral gap, on different intersecting ranges. After optimising, we get the global result.

\begin{coro}[Cauchy upper bounds]\label{cor:uppbound}
For $n\geq2$, we have the following spectral gap upper bound: 
\[\lambda_1(-L) \leq\left\{
\begin{array}{ll} 
(\beta-n/2)^2 & \text{if}\quad n/2<\beta\leq n/2+2\\
4(\beta-n/2-1)& \text{if}\quad n/2+2\leq\beta\leq n+1\\ 
2(\beta-1) & \text{if}\quad n+1\leq \beta.
\end{array}\right.\]
\end{coro}
%%%%%%%%%%%%%%%%%%%%%%%%%%%%%%%%%%%%%%%%%%%%%%%%%%%%%%%%%%%%%%%%%%%%%%%%%%%%
%%%%%%%%%%%%%%%%%%%%%%%%%%%%%%%%%%%%%%%%%%%%%%%%%%%%%%%%%%%%%%%%%%%%%%%%%%%%
\section{Weighted Poincaré inequalities}
\label{Sec:wPI}
Let $\omega$ be a smooth positive function on $M$. We assume that there exists $\beta>0$ such that the measure with density $\omega^{-\beta}$ with respect to the Riemannian volume measure, denoted $\mu_\beta$, is a probability measure. Up to additional assumptions on~$\omega$ we will be able to determine on which range of $\beta$ this measure assumption is satisfied.

The goal of this  section is to obtain general weighted Poincaré inequalities for the probability $\mu_\beta$, under additional assumptions on $\omega$ such as strong convexity or bounded convexity. These inequalities will be sharp in the sense that for generalised Cauchy measures, the weight and the constant are optimal. The main innovations of this section is the Poincaré inequality on a range of  small parameters $\beta\leq n+1$, beyond the criterion of convex measures. In particular, when applied to generalised Cauchy measures, the integrated curvature-dimension approach allows to recover the exact spectral gap even for $n/2<\beta\leq n+1$.

The measure $\mu_\beta$ is associated with the symmetric diffusion operator, defined on $\CC^\infty_c(M)$ by
\[Lf = \omega \Delta f -(\beta-1)\langle d\omega, df\rangle.\]
This special operator is chosen such that its carré du champ gives the weight $\omega$ in the Poincaré-type inequalities. Indeed, we have the following result:

\begin{prop}\label{prop:gamma}
For all $f\in\CC_c^\infty(M)$, we have
\begin{align*}
\Gamma(f) =& \omega|df|^2,\\
\Gamma_2(f) =& \|\omega\Hess(f)\|^2_{HS} +\omega^2\Ric(\nabla f,\nabla f) +\frac{1}{2}\left[\omega\Delta\omega-(\beta-1)|d\omega|^2\right]|df|^2\\
&\quad +\langle d|df|^2, \omega d\omega\rangle - \langle\Delta f df, \omega d\omega\rangle + (\beta-1)\omega\Hess(\omega)(\nabla f, \nabla f).
\end{align*}
\end{prop}

\begin{proof}
The computation of $\Gamma$ does not present any difficulties. We will only develop the computation of $\Gamma_2$. For $f\in\CC_c^\infty(M)$, we have
\begin{align*}
2\Gamma_2(f)
&= L\Gamma(f)-2\Gamma(f,Lf)\\
&= \omega\Delta(\omega|df|^2) - (\beta-1)\langle d\omega, d(\omega|df|^2)\rangle -2\omega\left\langle df, d(\omega\Delta f-(\beta-1)\langle d\omega, df\rangle)\right\rangle\\
&= \omega\Delta\omega|df|^2 + 2\omega\langle d\omega, d|df|^2\rangle +\omega^2\Delta|df|^2 -(\beta-1)|d\omega|^2|df|^2-(\beta-1)\langle d|df|^2,\omega d\omega\rangle\\
&\quad -2\langle \Delta f df, \omega d\omega\rangle -2\omega^2\langle df, d|df|^2\rangle +2(\beta-1) \omega\left\langle df, d\langle df, d\omega\rangle \right\rangle.
\end{align*} 

Then, we use Bochner's formula for the Laplacian 
\begin{equation}\label{eq:boch}
\Delta |df|^2 -2\langle df, d\Delta f\rangle = 2\|\Hess(f)\|^2_{HS}+2\Ric(\nabla f,\nabla f).
\end{equation}

Together with the remark that $\Hess(f)(\nabla f, \cdot) = \frac{1}{2}d|df|^2$, this ends the proof.
\end{proof} 

Let us note that, as $\omega$ is positive, then the carré du champ cancels on constant functions only. This means that the diffusion is ergodic and Proposition \ref{prop:PI} can be applied.
 
In this general setting, it remains unclear whereas $L$ satisfies or not a curvature-dimension criterion. Yet, in the special case of generalised Cauchy measures, we have a positive answer. First, we specify the $\Gamma_2$ operator associated with generalised Cauchy measures.

\begin{coro}[Cauchy carré du champ]\label{prop:GammaC}
Let $\omega : x\in\R^n\mapsto 1+|x|^2$. For all $f\in\CC^\infty_c(\R^n)$ we have
\begin{align*}
\Gamma(f) &= \omega|df|^2,\\
\Gamma_2(f) &= \|\omega\Hess(f)\|^2_{HS} + [n\omega+2(\beta-1)]|df|^2 + 2\langle d|df|^2, \omega x\rangle -2\langle\Delta f df, \omega x\rangle.
\end{align*}
\end{coro}

This formula can be factorised so as to obtain a $CD(0,\infty)$ criterion and to prove the optimality of this criterion. This property was already proved in \cite{Hug} in dimension $n=2$.

\begin{prop}[Cauchy $CD$ criterion]\label{prop:cd0}
Let $\omega : x\in\R^n\mapsto 1+|x|^2$. For $n\geq 1$, $L$ satisfies $CD(0,\infty)$. Moreover, for every $\rho>0$, it does not satisfy $CD(\rho,\infty)$.
\end{prop}
\begin{proof}
The idea of this proof is to factorise every second-order derivative terms as the squared Hilbert-Schmidt norm of a more intricate operator. For all $f\in\CC_c^\infty(\R^n)$, we have
\begin{align*}
\Gamma_2(f) 
&= \|\omega\Hess(f)\|^2_{HS}  + 2\langle d|df|^2, \omega x\rangle -2\langle\Delta f df, \omega x\rangle + [n\omega+2(\beta-1)]|df|^2\\
&= \|\omega\Hess(f)\|^2_{HS}  + 4\omega\Hess(f)(\nabla f, x) -2\omega \Trace(\Hess(f))\langle df, x\rangle + [n\omega+2(\beta-1)]|df|^2\\
&= \left\|\omega\Hess(f) + x\otimes \nabla f +\nabla f\otimes x - \langle df, x\rangle\id\right\|^2_{HS}-\left\|x\otimes \nabla f +\nabla f\otimes x - \langle df, x\rangle\id\right\|^2_{HS}\\
&\quad+ [n\omega+2(\beta-1)]|df|^2\\
&= \left\|\omega\Hess(f) + x\otimes \nabla f +\nabla f\otimes x - \langle df, x\rangle\id\right\|^2_{HS} +(n-2)\left[|df|^2|x|^2-\langle df, x\rangle^2\right]\\
&\quad  + (2\beta+n-2)|df|^2.\\
\end{align*}
Let us notice that in dimension $n=1$ the second term vanishes. Then, these three terms are non-negative and so $L$ satisfies the $CD(0,\infty)$ criterion. 

The optimality comes from the comparison with the special function $f : x\in\R^n\mapsto 1/2 \ln(|x|^2)\xi(x)$, where $\xi$ is a  cutoff function, which is equal to $1$ on a neighbourhood of  $x_0\in\R^n$, such that $f$ belongs to $\CC^\infty_c(\R^n)$. For every $x$ on this neighbourhood, we have
\[\Gamma(f)(x) = 1+\frac{1}{|x|^2},\quad \Gamma_2f(x) = \frac{n}{|x|^4} + \frac{2\beta+n-2}{|x|^2}.\]
Choosing $|x_0|$ large enough, it follows that for every $\rho>0$, there exists a function $f\in\CC_c^\infty(\R^n)$ and $x\in\R^n$ such that $\Gamma_2(f)(x) <\rho\Gamma(f)(x)$.
\end{proof}

Let us remark that the symmetric term in the Hilbert-Schmidt norm is necessary to obtain this optimal result. The same computation with a term 
\[\|\omega\Hess(f) + x\otimes\nabla f-\langle df, x\rangle\id\|^2_{HS},\]
would not have been sufficient. This phenomenon will be encountered again while dealing with the lower range of $\beta$. Also, this proof highly relies on the special form of $\omega$ and its generalisation to other $\omega$ is not trivial.

Thus, for generalised Cauchy measures at least, we cannot obtain any $CD(\rho,\infty)$ criterion for positive $\rho$. This is a reason why we are interested in integrated criterion. Before going further, here are the following integration by parts formulae:

\begin{lemme}
For all $f\in\CC_c^\infty(M)$, we have
\begin{equation}\label{eq:IPP1}
\int_M\langle d|df|^2, \omega d\omega\rangle\, d\mu_\beta = \int_M \left[-\omega\Delta\omega +(\beta-1)|d\omega|^2\right]|df|^2\, d\mu_\beta,
\end{equation}
\begin{equation}\label{eq:IPP2}
\begin{split}
\int_M \langle\Delta f df, \omega d\omega\rangle\, d\mu_\beta =& \int_M -\frac{1}{2}\langle d|df|^2, \omega d\omega\rangle -\omega\Hess(\omega)(\nabla f,\nabla f)\\
&\quad +(\beta-1)\langle df, d\omega\rangle^2\, d\mu_\beta.
\end{split}
\end{equation}
Furthermore, if $\beta\neq2$, we have
\begin{equation}\label{eq:IPP3}
\int_M\langle d|df|^2, \omega d\omega\rangle\, d\mu_\beta = \frac{1}{\beta-2} \int_M \omega^2\Delta|df|^2\, d\mu_\beta,
\end{equation}
\begin{equation}\label{eq:IPP4}
\int_M \langle\Delta f df, \omega d\omega\rangle\, d\mu_\beta = \frac{1}{\beta-2}\int_M \omega^2\langle df, d\Delta f\rangle + (\omega\Delta f)^2\, d\mu_\beta.
\end{equation}
\end{lemme}

\begin{proof}
For the first two formulae, we lower the differentiation degree on $f$. For all $f\in\CC_c^\infty(M)$, we have
\begin{align*}
\int_M\langle d|df|^2, \omega d\omega\rangle\, d\mu_\beta
=& \int_M \langle d|df|^2, \omega^{1-\beta} d\omega\rangle\, d\vol\\
=& -\int_M |df|^2\Div(\omega^{1-\beta} d\omega)\, d\vol\\
=& -\int_M |df|^2\left[\omega^{1-\beta} \Div(d\omega)+(1-\beta)\omega^{-\beta}|d\omega|^2\right]\, d\vol\\
=& \int_M \left[-\omega\Delta\omega +(\beta-1)|d\omega|^2\right]|df|^2\, d\mu_\beta.
\end{align*}

\begin{align*}
\int_M \langle\Delta f df, \omega d\omega\rangle\, d\mu_\beta
=& \int_M \Delta f\langle df, \omega^{1-\beta} d\omega\rangle\, d\vol\\
=& -\int_M \left\langle df, d\langle df, \omega^{1-\beta} d\omega\rangle\right\rangle\, d\vol\\
=& -\int_M \omega^{1-\beta}\Hess(f)(\nabla f, \nabla \omega) + \omega^{1-\beta}\Hess(\omega)(\nabla f, \nabla f)\\
&\quad +(1-\beta)\omega^{-\beta}\langle df, d\omega\rangle^2\, d\vol\\
=& \int_M -\frac{1}{2}\langle d|df|^2, \omega d\omega\rangle -\omega\Hess(\omega)(\nabla f,\nabla f)\\
&\quad +(\beta-1)\langle df, d\omega\rangle^2\, d\mu_\beta.
\end{align*}

Now, for $\beta\neq2$, we can also increase the differentiation degree on $f$. Then, we have

\begin{align*}
\int_M\langle d|df|^2, \omega d\omega\rangle\, d\mu_\beta 
=&\int_M\langle d|df|^2, \omega^{1-\beta} d\omega\rangle\, d\vol\\ 
=&\int_M\langle d|df|^2, \frac{d\omega^{2-\beta}}{2-\beta} \rangle\, d\vol\\ 
=&-\int_M\Delta|df|^2 \frac{\omega^{2-\beta}}{2-\beta} \rangle\, d\vol\\ 
=& \frac{1}{\beta-2} \int_M \omega^2\Delta|df|^2\, d\mu_\beta.
\end{align*}

\begin{align*}
\int_M \langle\Delta f df, \omega d\omega\rangle\, d\mu_\beta 
=& \int_M \langle\Delta f df, \omega^{1-\beta} d\omega\rangle\, d\vol\\
=& \int_M \langle\Delta f df, \frac{d\omega^{2-\beta}}{2-\beta}\rangle\, d\vol\\
=& -\int_M\Div(\Delta f df)\frac{d\omega^{2-\beta}}{2-\beta}\, d\vol\\
=& -\int_M\left[\langle d\Delta f, df\rangle+ \Delta f\Div(df)\right]\frac{d\omega^{2-\beta}}{2-\beta}\, d\vol\\
=& \frac{1}{\beta-2}\int_M \omega^2\langle df, d\Delta f\rangle + (\omega\Delta f)^2\, d\mu_\beta.
\end{align*}
\end{proof}

These formulae allow us to establish a first simplification.

\begin{lemme}\label{prop:gammabis}
For all $f\in\CC_c^\infty(M)$, we have
\begin{align*}
\int_M \Gamma_2(f)\, d\mu_\beta
=& \int_M \|\omega\Hess(f)\|^2_{HS} +\omega^2\Ric(\nabla f,\nabla f) + (\beta-1)\omega\Hess(\omega)(\nabla f, \nabla f)\\
&\quad +\frac{1}{2}\langle d|df|^2, \omega d\omega\rangle - \langle\Delta f df, \omega d\omega\rangle\, d\mu_\beta .
\end{align*}
\end{lemme}

\begin{proof}
\begin{align*}
\int_M \Gamma_2(f)\, d\mu_\beta
=& \int_M \|\omega\Hess(f)\|^2_{HS} +\omega^2\Ric(\nabla f,\nabla f) + (\beta-1)\omega\Hess(\omega)(\nabla f, \nabla f)\\
&\quad +\frac{1}{2}\left[\omega\Delta\omega-(\beta-1)|d\omega|^2\right]|df|^2 +\frac{1}{2}\underbrace{\langle d|df|^2, \omega d\omega\rangle}_{(i)}\\
&\quad +\frac{1}{2}\langle d|df|^2, \omega d\omega\rangle - \langle\Delta f df, \omega d\omega\rangle\, d\mu_\beta.
\end{align*}
The result is achieved by using the formula \eqref{eq:IPP1} on the term $(i)$.
\end{proof}

This last formula for $\int_M \Gamma_2\, d\mu_\beta$ can be divided into four parts. The first one is a squared norm of the Hessian and is non-negative. The second one, gathering the Ricci tensor and the $\Hess(\omega)$, is a kind of Bakry-\'Emery curvature, upon which, we will probably need some additional assumption. The last two terms are more problematical. In the sequel, we will deal with these terms using appropriate combinations of our integration by parts formulae, depending on the parameter $\beta$.

\subsection{Strong convexity}
Firstly, we recover a general weighted Poincaré inequality. It was proved in the Euclidean space $\R^n$ in \cite{Ngu} (Theorem 12) and in Riemannian manifolds with non-negative Ricci curvature in \cite{GZ}. This classical result is only valid for convex measures with $\beta\geq n+1$ and under strong convexity property. Here, we complete this result and provide a criterion for general Riemannian manifolds.

\begin{prop}\label{prop:grg}
For all $\beta\neq2$ and $f\in\CC^\infty_c(M)$, we have
\begin{align*}
\int_M\Gamma_2(f)\, d\mu_\beta = \int_M& \frac{\beta-(n+1)}{\beta-2}\|\omega\Hess(f)\|^2_{HS}+ \frac{n}{\beta-2}\left[\|\omega\Hess(f)\|^2_{HS}-\frac{1}{n}(\omega\Delta f)^2\right]\\ 
&+(\beta-1)\omega\left[\frac{1}{\beta-2}\omega\Ric +\Hess(\omega)\right](\nabla f,\nabla f)\, d\mu_\beta.
\end{align*}
\end{prop}

\begin{proof}
From Lemma \ref{prop:gammabis}, we have
\begin{align*}
\int_M\Gamma_2(f)\, d\mu_\beta =& \int_M \|\omega\Hess(f)\|^2_{HS} +\omega^2\Ric(\nabla f,\nabla f)\\
&\quad+\frac{1}{2}\underbrace{\langle d|df|^2, \omega d\omega\rangle}_{(i)} -\underbrace{\langle\Delta f df, \omega d\omega\rangle}_{(ii)} +(\beta-1)\omega\Hess(\omega)(\nabla f, \nabla f)\\
=&\int_M \|\omega\Hess(f)\|^2_{HS} +\omega^2\Ric(\nabla f,\nabla f)+ (\beta-1)\omega\Hess(\omega)(\nabla f, \nabla f) \\
&\quad+\frac{\omega^2}{2(\beta-2)}(\Delta|df|^2-2\langle df, d\Delta|df|^2\rangle)-\frac{\omega^2}{\beta-2}(\Delta f)^2,\\
\end{align*}
where we used the equalities \eqref{eq:IPP3} and \eqref{eq:IPP4} on $(i)$ and $(ii)$ respectively. Using again Bochner's formula \eqref{eq:boch}, we obtain the result.
\end{proof}

Let us note that this expression of $\int_M\Gamma_2(f)\, d\mu$ can be found in \cite{Ngu}, without the use of the carré du champ formalism.

The previous proposition suggests some conditions under which we can obtain a spectral gap inequality. The more important point to understand is the condition upon the Bakry-\'Emery like-curvature. A first assumption can be set as follows. Let us assume that $\beta\neq2$ and that there exists a $c>0$ such that
\begin{equation}\label{eq:bdlow}
\frac{\omega}{\beta-2}\Ric+\Hess(\omega)\geq c \id.\tag{H1}
\end{equation}
This assumption is not very convenient because it depends on $\beta$ and because this dependence is not explicit. Yet, in the Euclidean space, the assumption is equivalent to strong convexity on $\omega$: it exists~${\rho_->0}$ such that 
\begin{equation}\label{eq:bdlow2}
\Hess(\omega)\geq \rho_-  \id.
\end{equation}
Let us remark that under this strong convexity assumption, the measure $d\mu_\beta$ is a probability measure for every~${\beta>n/2}$. The hypothesis \eqref{eq:bdlow2} is also sufficient in a Riemannian manifold with non-negative Ricci curvature. We can now state our first result.

\begin{theorem}\label{prop:ggg}
For $n\geq2$ and $\beta\geq n+1$, under assumption \eqref{eq:bdlow}, we have $\lambda_1(-L)\geq c(\beta-1)$.
For $n=1$ and $\beta>1$, under assumption \eqref{eq:bdlow}, we have $\lambda_1(-L)\geq c(\beta-1)$.
\end{theorem}

\begin{proof}
Let us begin with the general case $n\geq2$. We must remark that for $n\geq 2$ and $\beta\geq n+1$ the assumption $\beta\neq2$ is redundant as we have $\beta>2$. Then, under the assumption \eqref{eq:bdlow} we have proved that
\begin{align*}
\int_M\Gamma_2(f)-c(\beta-1)\Gamma(f)\, d\mu_\beta \geq& \int_M \frac{\beta-(n+1)}{\beta-2}\|\omega\Hess(f)\|^2_{HS}\\
&\quad +\frac{n}{\beta-2}\left[\|\omega\Hess(f)\|^2_{HS}-\frac{1}{n}(\omega\Delta f)^2\right]\, d\mu_\beta.
\end{align*}
The numerical coefficients are non-negative and the Cauchy-Schwarz inequality implies that for all linear operator $A$ acting on the tangent bundle $TM$, we have
\[\|A\|^2_{HS}\geq \frac{1}{n}(\Trace A)^2.\]
Applied to $A=\omega\Hess(f)$, it proves an integrated curvature-dimension criterion. We conclude with Proposition \ref{prop:PI}.

In the particular case $n=1$, the two problematic terms from Lemma \ref{prop:gammabis} offset each other. Under assumption \eqref{eq:bdlow}, we get 
\[\int_M\Gamma_2(f)-c(\beta-1)\Gamma(f)\, d\mu_\beta \geq \int_M \|\omega\Hess(f)\|^2_{HS}\, d\mu_\beta.\]
\end{proof}

In this general case, we only obtain an upper bound of the error term although a lower bound would have been more interesting. In the case of the generalised Cauchy measures, the potential $\omega$ satisfies $\Hess(\omega) = 2\id$. Then we can apply our theorem to get a spectral gap lower bounds together with an exact expression of the error term. 

\begin{coro}[Cauchy - Upper range]\label{cor:up}
Let $\omega : x\in\R^n\mapsto 1+|x|^2$. For all $n\geq2$ and $\beta\geq n+1$, we have~$\lambda_1(-L) \geq 2(\beta-1)$, 
and for all $f\in\CC_c^\infty(\R^n)$, we have
\begin{align*}
&2(\beta-1)\Var_{\mu_\beta}(f) -\int_{\R^n}|df|^2(1+|x|^2)\, d\mu_\beta\\
& = -2\int_0^{+\infty}\int_{\R^n}\frac{\beta-(n+1)}{\beta-2}\|\omega\Hess(\BP_tf)\|^2_{HS} +\frac{n}{\beta-2}\left[\|\Hess(\omega\BP_tf)\|^2_{HS}-\frac{1}{n}(\omega\Delta \BP_tf)^2\right]\, d\mu_\beta\, dt.
\end{align*}
\end{coro}

As explained in Section \ref{Sec:var}, our expression of the error term characterises the extremal function. For $\beta>n+1$, an extremal function $f$ must satisfy \[\|\Hess(f)\|_{HS}=0.\] Then it is an affine function. This is coherent as, in this range, linear functions are eigenfunctions associated with $2(\beta-1)$. In the particular case $\beta=n+1$, an extremal function $f$ must satisfy  \[|\Hess(f)\|^2_{HS}-\frac{1}{n}(\Delta f)^2.\] Then in this case, $f$ is quadratic. This is coherent as, for this parameter, linear functions and $x\in\R^n\mapsto |x|^2-\frac{n}{2\beta-n-2}$ are eigenfunctions associated with $2n$. 

\subsection{Upper bounded convexity}
To obtain Proposition \ref{prop:grg}, we kept an Hessian term. This term is cancelled by affine functions which are extremal functions for generalised Cauchy measures for $\beta\geq n+1$. However, we have seen that the extremal functions for $\beta\leq n+1$ are no longer linear. This suggests a formula for $\int_M\Gamma_2\, d\mu$ without the Hessian term. On the other hand, the term, specific to the dimension $n\geq2$ 
\[\|\omega\Hess(f)\|^2_{HS}-\frac{1}{n}(\omega\Delta f)^2,\]
is a very good candidate as it is cancelled by the extremal functions for generalised Cauchy measures on this range. This idea provides a Poincaré inequality for general $\omega$ and $\beta\leq n+1$. This result for non-convex measures is an important innovation of our work. Nonetheless, it requires a stronger assumption than the strong convexity. 

\begin{prop}\label{prop:irg}
For all $n\geq2$, $\beta>n/2$ and $f\in\CC^\infty_c(M)$, we have
\begin{align*}
\int_M\Gamma_2(f)\, d\mu_\beta = \int_M  &\frac{n}{n-1}\left[\|\omega\Hess(f)\|^2_{HS}-\frac{1}{n}(\omega\Delta f)^2\right]\\
& +\frac{(n+1-\beta)(\beta-1)}{n-1}\left[|df|^2|d\omega|^2-\langle df,d\omega\rangle^2\right]\\ 
&+\omega\left[\frac{n}{n-1}\omega\Ric+(\beta-1)\Hess(\omega)\right](\nabla f,\nabla f)\\
&+\frac{n+1-\beta}{n-1}\omega\left[\Hess(\omega)-\Delta\omega\id\right](\nabla f,\nabla f)\, d\mu_\beta.
\end{align*}
\end{prop}

\begin{proof}
We begin with the generic case $\beta\neq2$. Let $\lambda\in\R$ be a parameter to be optimised later. From Lemma \ref{prop:gammabis}, for all $f\in\CC^\infty_c(M)$, we have 
\begin{align*}
\int_M\Gamma_2(f)\, d\mu_\beta 
=&\int_M \|\omega\Hess(f)\|_{HS}^2+\omega^2\Ric(\nabla f,\nabla f)+ (\beta-1)\omega\Hess(\omega)(\nabla f, \nabla f)\\
&\quad +\frac{\lambda}{2}\langle d|df|^2, \omega d\omega\rangle - \lambda\langle\Delta f df, \omega d\omega\rangle \\
&\quad +\frac{1-\lambda}{2}\underbrace{\langle d|df|^2, \omega d\omega\rangle}_{(i)}  - (1-\lambda)\underbrace{\langle\Delta f df, \omega d\omega\rangle}_{(ii)}\, d\mu_\beta\\
=&\int_M \|\omega\Hess(f)\|_{HS}^2 +\frac{1-\lambda}{\beta-2}\left[\|\omega\Hess(f)\|_{HS}^2-(\omega\Delta f)^2\right]\\
&\quad + \left(1+\frac{1-\lambda}{\beta-2}\right)\omega^2\Ric(\nabla f, \nabla f) +(\beta-1)\omega\Hess(\omega)(\nabla f, \nabla f)\\
&\quad +\frac{\lambda}{2}\underbrace{\langle d|df|^2, \omega d\omega\rangle}_{(iv)} - \lambda\underbrace{\langle\Delta f df, \omega d\omega\rangle}_{(v)}\, d\mu_\beta.
\end{align*}
Here, we have used the formulae \eqref{eq:IPP3} and \eqref{eq:IPP4} for $(i)$ and $(ii)$ respectively and together with Bochner's formula. Then, we choose $\lambda$ in order to exactly compensate the term~$\|\Hess(f)\|^2_{HS}$ with~$\frac{1}{n}(\Delta f)^2$. It results that
\[\lambda = \frac{n+1-\beta}{n-1}.\]
To conclude, we use the formulae \eqref{eq:IPP1} and \eqref{eq:IPP2} to treat the terms $(iv)$ and $(v)$.

In the specific case, $\beta=2$, we do not need the forbidden formulae \eqref{eq:IPP3} and \eqref{eq:IPP4}. Indeed, thanks to Bochner's formula, we have
\begin{align*}
\int_M (\omega\Delta f)^2\, d\mu_2
=&\int_M \Delta f \Delta f\, d\vol\\
=&-\int_M\langle df, d\Delta f\rangle\, d\vol\\
=&-\int_M\langle df, d\Delta f\rangle-\frac{1}{2}\Delta|df|^2\, d\vol\\
=&\int_M\|\Hess(f)\|^2_{HS}+\Ric(\nabla f, \nabla f)\, d\vol\\
=&\int_M\|\omega\Hess(f)\|^2_{HS}+\omega^2\Ric(\nabla f, \nabla f)\, d\mu_2.\\
\end{align*}
It follows that
\[\int_M\|\omega\Hess(f)\|^2_{HS}\, d\mu_2
=\int_M \frac{n}{n-1}\left[\|\omega\Hess(f)\|^2_{HS}-\frac{1}{n}(\omega\Delta f)^2\right]+\frac{1}{n-1}\Ric(\nabla f, \nabla f)\, d\mu_2.\]
The proof ends as in the generic case.
\end{proof}

In order to derive some spectral gap inequality, for~${\beta<n+1}$, the assumption \eqref{eq:bdlow} is not sufficient any more. In this case, we make the following assumption:  it exists $\tilde{c}>0$ such that 
\begin{equation}\label{eq:bdup}
\frac{n}{n-1}\omega\Ric+(\beta-1)\Hess(\omega)+\frac{n+1-\beta}{n-1}\omega[\Hess(\omega)-\Delta\omega\id]\geq \tilde{c}\id. \tag{H2}
\end{equation}
In a general Riemannian manifold, it seems difficult to obtain an handier condition. However, in the Euclidean space $\R^n$, this condition is equivalent to boundedness of the Hessian operator: it exists~${0<\rho_-\leq\rho_+}$ such that
\begin{equation}\label{eq:bdup2}
\rho_-\id \leq \Hess(\omega) \leq \rho_+\id.
\end{equation}
In this case, we will denote by $\kappa$ the corresponding condition number:
\[\kappa=\frac{\rho_+}{\rho_-}\geq1.\]
As previously, in a Riemannian manifold with non-negative Ricci curvature, we can also use the simpler boundedness assumption \eqref{eq:bdup2}. Even if it is more restrictive than \eqref{eq:bdup}, it is explicit and independent of $\beta$. 

Let us remark that if the Hessian is uniformly bounded, then $\mu$ is a probability on $\R^n$ if and only if~${\beta>n/2}$.

\begin{theorem}\label{prop:igg}
For $n\geq2$ and $\beta\leq n+1$, under assumption \eqref{eq:bdup}, we have $\lambda_1(-L)\geq \tilde{c}$. Furthermore, if $\Ric\geq0$, then, under assumption \eqref{eq:bdup2}, we have 
\[\lambda_1(-L)\geq \rho_-\left(\beta-1-\frac{n+1-\beta}{n-1}(n\kappa-1)\right),\quad \forall\, \beta\in\left[\frac{n(n+1)\kappa-2}{n(\kappa+1)-2}, n+1\right].\]
\end{theorem}

\begin{proof}
The first general result is clear, using Cauchy-Schwarz inequality. Then, we remark that ${\beta-1+\lambda\geq 0}$ and $\lambda\geq0$. Under assumption \eqref{eq:bdup2}, we have
\[(\beta-1+\lambda)\Hess(\omega)-\lambda\Delta\omega\id
\geq [\rho_-(\beta-1+\lambda)-\lambda n\rho_+]\id
\geq \rho_-[\beta-1-\lambda(n\kappa-1)]\id.\]
The term $\rho_-[\beta-1-\lambda(n\kappa-1)]$ is positive if and only if $\frac{n(n+1)\kappa-2}{n(\kappa+1)-2}<\beta$.
\end{proof}

The assumption \eqref{eq:bdup2} implies the assumption \eqref{eq:bdlow2}, then it is interesting to remark that the associated spectral gap lower bounds are continuous at $\beta=n+1$. As the condition number $\kappa$ is always greater than $1$, we have
\[\frac{n(n+1)\kappa-2}{n(\kappa+1)-2}\geq \frac{n}{2}+1.\]
Hence, our two spectral gap lower bounds do not cover the entire range $]n/2, +\infty[$ admissible under assumption \eqref{eq:bdlow2}. The equality $\kappa=1$ is attained for generalised Cauchy measures. Then we have the following result.

\begin{coro}[Cauchy - Intermediate range]\label{cor:mid}
Let $\omega : x\in\R^n\mapsto 1+|x|^2$. For all $n\geq2$ and $n/2+1<\beta\leq n+1$, we have~${\lambda_1(-L) \geq 4(\beta-n/2-1)}$, 
and for all~$f\in\CC_c^\infty(\R^n)$, we have
\begin{align*}
4(\beta-n/2-1)&\Var_{\mu_\beta}(f) -\int_{\R^n}|df|^2(1+|x|^2)\, d\mu_\beta \\ 
= -2\int_0^{+\infty}\int_{\R^n}&\frac{n}{n-1}\left[\|\omega\Hess(\BP_tf)\|^2_{HS} -\frac{1}{n}(\omega\Delta \BP_tf)^2\right]\\
 &+4(\beta-1)\frac{n+1-\beta}{n-1}\left[|d\BP_tf|^2|x|^2-\langle d\BP_tf, x\rangle^2\right]\, d\mu_\beta\, dt.
\end{align*}
\end{coro}

For $n/2+2<\beta<n+1$, an extremal function must satisfy $|d f|^2|x|^2 = \langle d f, x\rangle^2$.
Then there exist $a,b\in\R$ such that for all $x\in\R^n$, $f(x) = a|x|^2+b$. For $\beta = n+1$, we recover the same error term as in the previous subsection and so, the same extremal function. The parameter $\beta=n+1$ marks a transition between affine extremal functions and quadratic ones.

This result for mid-range $\beta$ improves the previous results from the literature as it is the first one which obtains the optimal spectral gap with curvature-dimension argument.
 
\subsection{Lower range}
It seems difficult to obtain any result for the smallest $\beta$ and for unspecified $\omega$. Yet, in the particular case of generalised Cauchy measures, we will be able to conclude. As for the one-dimensional case, the appropriate formula for the lower range is obtained by factorising the second-order derivative, in the way we proved the curvature-dimension criterion in Proposition \ref{prop:cd0}. However, we need to factorise a fraction of it only and use integration parts formulae on the remaining second-order terms. The result is achieved by optimising on these fractions.  

In the previous subsection, we have obtained the following formula
\begin{equation}\label{eq:g2mid}
\begin{split}
\int_{\R^n}\Gamma_2(f)\, d\mu_\beta = \int_{\R^n} &\frac{n}{n-1}\left[\|\omega\Hess(f)\|^2_{HS} -\frac{1}{n}(\omega\Delta f)^2\right]\\
& +4(\beta-1)\frac{n+1-\beta}{n-1}\left[|df|^2|x|^2-\langle df, x\rangle^2\right]+ 4(\beta-\frac{n}{2}-1)\Gamma(f)\, d\mu_\beta.
\end{split}
\end{equation}

Let $\varepsilon$ be the optimisation parameter. We want to force the following factorisation   
\begin{align*}
\left\|\omega\Hess f + \varepsilon\frac{\nabla f\otimes x+ x\otimes\nabla f}{2}\right\|^2_{HS}
&= \|\omega\Hess f\|^2_{HS} + \varepsilon\langle d|df|^2, x\omega\rangle + \frac{\varepsilon^2}{2}[|df|^2|x|^2+\langle df, x\rangle^2]\\
\left(\omega\Delta f + \varepsilon\langle df, x\rangle\right)^2 &= \left(\omega\Delta f\right)^2 + 2\varepsilon\langle \Delta fdf, x\omega\rangle +\varepsilon^2\langle df, x\rangle^2.
\end{align*}

We inject these twisted terms in equation \eqref{eq:g2mid} and we use the integration by parts formulae \eqref{eq:IPP2} and then \eqref{eq:IPP1}. We have

\begin{align*}
\int_{\R^n}\Gamma_2(f)\, d\mu_\beta 
= \int_{\R^n}& \frac{n}{n-1}\left[\left\|\omega\Hess f + \varepsilon\frac{\nabla f\otimes x+ x\otimes\nabla f}{2}\right\|^2_{HS} -\frac{1}{n}\left(\omega\Delta f + \varepsilon\langle df, x\rangle\right)^2\right]\\
&-\frac{n\varepsilon}{n-1}\langle d|df|^2,x\omega\rangle -\frac{n\varepsilon^2}{2(n-1)}\left[|df|^2|x|^2+\langle df, x\rangle^2\right]\\
&+\frac{2\varepsilon}{n-1}\langle \Delta f df, \omega x\rangle +\frac{\varepsilon^2}{n-1}\langle df, x\rangle^2\\
& +4(\beta-1)\frac{n+1-\beta}{n-1}\left[|df|^2|x|^2-\langle df, x\rangle^2\right]+ 4(\beta-\frac{n}{2}-1)\Gamma(f)\, d\mu_\beta\\
= \int_{\R^n}& \frac{n}{n-1}\left[\left\|\omega\Hess f + \varepsilon\frac{\nabla f\otimes x+ x\otimes\nabla f}{2}\right\|^2_{HS} -\frac{1}{n}\left(\omega\Delta f + \varepsilon\langle df, x\rangle\right)^2\right]\\
& + \BB_\varepsilon\left[|df|^2|x|^2-\langle df, x\rangle^2\right] + \frac{4(\beta-1)\varepsilon-(n-1)\varepsilon^2}{n-1}|df|^2|x|^2\\
&-\frac{\varepsilon}{n-1}\langle d|df|^2, x\omega\rangle + \left[4(\beta-\frac{n}{2}-1)-\frac{2\varepsilon}{n-1}\right]\Gamma(f)\, d\mu_\beta\\
= \int_{\R^n}& \frac{n}{n-1}\left[\left\|\omega\Hess f + \varepsilon\frac{\nabla f\otimes x+ x\otimes\nabla f}{2}\right\|^2_{HS} -\frac{1}{n}\left(\omega\Delta f + \varepsilon\langle df, x\rangle\right)^2\right]\\
& + \BB_\varepsilon\left[|df|^2|x|^2-\langle df, x\rangle^2\right] - \varepsilon[\varepsilon+2(\beta-1)]|df|^2|x|^2\\
& +\left[4(\beta-\frac{n}{2}-1)+(n-2)\varepsilon\right]\Gamma(f)\, d\mu_\beta\\
= \int_{\R^n}& \frac{n}{n-1}\left[\left\|\omega\Hess f + \varepsilon\frac{\nabla f\otimes x+ x\otimes\nabla f}{2}\right\|^2 -\frac{1}{n}\left(\omega\Delta f + \varepsilon\langle df, x\rangle\right)^2\right]\\
& +\BB_\varepsilon\left[|df|^2|x|^2-\langle df, x\rangle^2\right]+ \CC_\varepsilon|df|^2+\DD_\varepsilon\Gamma(f)\, d\mu_\beta,
\end{align*}

where $\BB$, $\CC$ and $\DD$ are defined by
\begin{align*}
\BB_\varepsilon &= \frac{(n-2)\varepsilon^2-8(\beta-1)\varepsilon+8(\beta-1)(n+1-\beta)}{2(n-1)},\\
\CC_\varepsilon &= \varepsilon[\varepsilon+2(\beta-1)],\\
\DD_\varepsilon &= -\varepsilon^2+(n+2-2(\beta-1))\varepsilon+4(\beta-\frac{n}{2}-1).
\end{align*}
Note that the particular form of $\omega$ has been used in several simplifications, as in this case $|d\omega|^2$ and $\omega\Hess(\omega)$ are easily comparable to $\omega$. To optimise this expression, we are looking for the maximum of $\DD$ over the parameters $\varepsilon$ such that $\BB$ and $\CC$ are non-negative. This optimum is reached for $\varepsilon_0 = n/2+2-\beta$, and we have 
\begin{align*}
\BB_{\varepsilon_0}&= \frac{(n-2)\left[4(\beta-1)^2-4(n-2)(\beta-1)+(n+2)^2\right]}{8(n-1)},\\
\CC_{\varepsilon_0}&= \left(\frac{n}{2}+2-\beta\right)\left(\beta+\frac{n}{2}\right),\\
\DD_{\varepsilon_0} &= \left(\beta-\frac{n}{2}\right)^2.
\end{align*}

Then, for all $n\geq2$ and $n/2<\beta$, we have the following formula: 

\begin{equation}\label{eq:g2low}
\begin{split}
\int_{\R^n}\Gamma_2(f)\, d\mu_\beta = \int_{\R^n}& \frac{n}{n-1}\left[\left\|\omega\Hess f + \varepsilon_0\frac{\nabla f\otimes x+ x\otimes\nabla f}{2}\right\|^2 -\frac{1}{n}\left(\omega\Delta f + \varepsilon_0\langle df, x\rangle\right)^2\right]\\
& +\frac{(n-2)\left[4(\beta-1)^2-4(n-2)(\beta-1)+(n+2)^2\right]}{8(n-1)}\left[|df|^2|x|^2-\langle df, x\rangle^2\right]\\
& +\left(\frac{n}{2}+2-\beta\right)\left(\beta+\frac{n}{2}\right)|df|^2 +\left(\beta-\frac{n}{2}\right)^2\Gamma(f)\, d\mu_\beta.
\end{split}
\end{equation}

Remark that the optimisation parameter $\varepsilon_0$ was expected, as we recover the expression \eqref{eq:g2mid} at the transition $\beta= n/2+2$. For $\beta\geq n/2+2$, we can prove that the optimal is $\varepsilon_0=0$.

To conclude, for all $n\geq2$ and $n/2<\beta\leq n/2+2$ we have $\BB_{\varepsilon_0}\geq0$, $\CC_{\varepsilon_0}\geq$ and $\DD_{\varepsilon_0}>0$. Hence, we have proved the following weighted Poincaré inequality on the lower range.

\begin{coro}[Cauchy - Lower range]\label{cor:low}
For all $n\geq2$ and $n/2<\beta\leq n/2+2$, we have~$\lambda_1(-L)\geq (\beta-n/2)^2$ and for all~${f\in\CC_c^\infty(\R^n)}$, we have
\begin{align*}
\left(\beta-\frac{n}{2}\right)^2\Var_{\mu_\beta}(f)& -\int_{\R^n}|df|^2(1+|x|^2)\, d\mu_\beta \\ 
 = -2\int_0^{+\infty}\int_{\R^n}&\frac{n}{n-1}\left\|\omega\Hess \BP_tf + \varepsilon_0\frac{\nabla \BP_tf\otimes x+ x\otimes\nabla \BP_tf}{2}\right\|^2 -\frac{1}{n-1}\left(\omega\Delta \BP_tf + \varepsilon_0\langle d\BP_tf, x\rangle\right)^2\\
& +\frac{(n-2)\left[4(\beta-1)^2-4(n-2)(\beta-1)+(n+2)^2\right]}{8(n-1)}\left[|d\BP_tf|^2|x|^2-\langle d\BP_tf, x\rangle^2\right]\\
& +\left(\frac{n}{2}+2-\beta\right)\left(\beta+\frac{n}{2}\right)|d\BP_tf|^2\, d\mu_\beta\, dt,
\end{align*}
where $\varepsilon_0=\beta-\frac{n}{2}-2$.
\end{coro}

Note that the error term is positive for all $n/2<\beta<n/2+2$. This implies that there cannot exist any extremal function.

When all has been said and done, we have obtained a global  lower bound for all $n/2<\beta$. Comparing to the upper bounds from Corollary \ref{cor:uppbound}, this bound is optimal and we have proved the $n$-dimensional part of Theorem \ref{prop:spec}.

\begin{coro}
For $n\geq2$, we have
\[\lambda_1(-L) =\left\{
\begin{array}{ll} 
(\beta-n/2)^2 & \text{if}\quad n/2<\beta\leq n/2+2\\
4(\beta-n/2-1)& \text{if}\quad n/2+2\leq\beta\leq n+1\\ 
2(\beta-1) & \text{if}\quad n+1\leq \beta.
\end{array}\right.\]
\end{coro}

It is noticeable that the obtained expressions are continuous in $\beta$. Let us remark that in dimension~${n=2}$, the intermediate range $n/2+2\leq \beta\leq n+1$ is reduced to~$\beta =3$. Our method shows that curvature-dimension arguments are sufficient to obtain the spectral gap for generalised Cauchy measures. In this work, the error term has been used only to recover extremal functions. It should be interesting to use these formulae in order to get some explicit lower bounds of the deficit in terms of the distance to extremal function in some topology, as it can be done for Gross' logarithmic Sobolev inequality in \cite{BGRS}.

%%%%%%%%%%%%%%%%%%%%%%%%%%%%%%%%%%%%%%%%%%%%%%%%%%%%%%%%%%%%%%%%%%%%%%%%%%%%
%%%%%%%%%%%%%%%%%%%%%%%%%%%%%%%%%%%%%%%%%%%%%%%%%%%%%%%%%%%%%%%%%%%%%%%%%%%%
\section{The one-dimensional case}
\label{Sec:1d}
In this section, we end the proof of Theorem \ref{prop:spec} with the dimension $n=1$. For the sake of completeness, we recover this result with our method and make some further remarks, especially on the error term and on extremal functions.   

We still denote by $\omega$ the function $\omega : x\in\R\mapsto 1+x^2$. As shown in Section \ref{Sec:wPI}, we have 
\begin{align*}
Lf &= \omega f'' - 2(\beta-1)xf',\\
\Gamma(f) &= \omega f'^2,\\
\Gamma_2(f) &= \omega^2f''^2 +[\omega+2(\beta-1)]f'^2 + 2f''f'x\omega.
\end{align*}

Firstly, we remark that, the operator $L$ satisfies a nice curvature-dimension criterion. Indeed, in the one-dimensional case, the factorisation from Proposition \ref{prop:cd0} brings
\[\Gamma_2(f)=(\omega f''+xf')^2+\frac{2\beta-1}{\omega}\Gamma(f).\]

This factorisation formula enlightens one of the two options we have. The first one is the integration by parts, as used in our general method. The second one, which appears here, is the factorisation. A mix between these two approaches allows to recover the spectral gap for the whole range $\beta\in]1/2,+\infty[$.

Let $\varepsilon\in\R$ be an optimisation parameter. We have 
\[\Gamma_2(f)= (\omega f'' +\varepsilon xf')^2 + 2(1-\varepsilon)f''f'x\omega + [2(\beta-1)+\omega-\varepsilon^2x^2]f'^2.\]

Now, using the integration by parts formula \eqref{eq:IPP1}, we get 

\[\int_\R \Gamma_2(f)\, d\mu_\beta = \int_\R (\omega f'' +\varepsilon xf')^2 + [A_\varepsilon + B_\varepsilon|x|^2]f'^2\, d\mu_\beta,\]
with $A_\varepsilon = 2(\beta-1)+\varepsilon$ and $B_\varepsilon = 2(\beta-1)+\varepsilon -\varepsilon(2(\beta-1)+\varepsilon)$. We optimise to obtain the maximum over $\varepsilon$ of $\min\{A_\varepsilon, B_\varepsilon\}$. Two cases appear: small $\beta$ and large $\beta$. 

If $\beta\geq 3/2$ then the optimum is reached for $\varepsilon=0$, as a direct application of Theorem \ref{prop:ggg}. We obtain 
\[\int_\R\Gamma_2(f)\,d\mu_\beta = \int_\R \omega^2 f''^2 +2(\beta-1)f'^2\, d\mu_\beta.\]
Then, according to Proposition \ref{prop:PI}, we have the following result.

\begin{coro}
If $\beta\geq3/2$, then $\lambda(-L)\geq 2(\beta-1)$ and for all $f\in\CC_c^\infty(\R)$, we have 
\[2(\beta-1)\Var_{\mu_\beta}(f) = \int_\R\Gamma(f)\, d\mu_\beta -2\int_0^{+\infty}\int_\R \omega^2(\BP_tf)''^2\, d\mu_\beta\, dt.\]
Furthermore, if $\beta>3/2$ then $\lambda(-L)= 2(\beta-1)$ as linear functions are eigenfunctions associated with $2(\beta-1)$.
\end{coro}

Actually, affine functions are the only extremal functions, up to regularity, as any extremal function, $f$,  satisfies $f''^2=0$. Remark that we cannot conclude yet when $\beta=3/2$ as affine functions do not belong to $\L^2(\mu_{3/2})$.
   
If $1/2<\beta\leq 3/2$, the maximum over $\varepsilon$ is obtained for $\varepsilon = 3/2-\beta$ and we have
\[\int_\R\Gamma_2(f)\,d\mu_\beta = \int_\R \left(\omega f''+\left(\frac{3}{2}-\beta\right)f'\right)^2 +\left(\beta-\frac{1}{2}\right)^2\omega f'^2 +\left(\beta-\frac{1}{2}\right)\left(\frac{3}{2}-\beta\right) f'^2\, d\mu_\beta.\]
Then, we obtain the final result.

\begin{coro}\label{cor:1low}
If $1/2<\beta\leq3/2$, then $\lambda(-L)=(\beta-1/2)^2$ and for all  
\begin{align*}
\left(\beta-\frac{1}{2}\right)^2\Var_{\mu_\beta}(f) - \int_\R\Gamma(f)\, d\mu_\beta = -2\int_0^{+\infty}\int_\R& \left(\omega\BP_tf''+\left(\frac{3}{2}-\beta\right)\BP_tf'\right)^2\\
& + \left(\frac{3}{2}-\beta\right)\left(\beta-\frac{1}{2}\right)\BP_tf'^2\, d\mu_\beta\, dt.
\end{align*} 
\end{coro}

Our computation only provides a lower bound of the spectral gap. On this range, there is no eigenfunction, as in the $n$-dimensional case. The article \cite{BJM2} proved the optimality using a sequence of functions.

\begin{prop}[\cite{BJM2}]\label{prop:1dopt}
For $n=1$, $1/2<\beta\leq3/2$ and $\varepsilon<(2\beta-3)/4$, the function $f_\varepsilon :x\mapsto x\omega^\varepsilon(x)$ belongs to $\L^2(\mu)$ and we have
\[\lim_{\varepsilon\uparrow(2\beta-3)/4} \frac{\int_\R\Gamma(f_\varepsilon)\,d\mu_\beta}{\Var_{\mu_\beta}(f_\varepsilon)}=\left(\beta-\frac{1}{2}\right)^2.\]
\end{prop}

While looking for an extremal function $f$ associated with $(\beta-1/2)^2$, we saw that, up to regularity, if it does exist, this function would satisfy the equation 
\[\omega f''+\left(\frac{3}{2}-\beta\right)f'=0.\]
The solutions of this equation are primitive functions of $\omega^{(2\beta-3)/4}$. This explains the choice of $f_\varepsilon$. Heuristically, the sequence $f_\varepsilon$ converges towards a function close to cancel the Hessian part but which is not in $\L^2(\mu_\beta)$.

%%%%%%%%%%%%%%%%%%%%%%%%%%%%%%%%%%%%%%%%%%%%%%
%% Support information, if any,             %%
%% should be provided in the                %%
%% Acknowledgements section.                %%
%%%%%%%%%%%%%%%%%%%%%%%%%%%%%%%%%%%%%%%%%%%%%%
\section*{Acknowledgements}
We would like to thank François Bolley for his encouragements and for our many useful discussions. We also thank Geneviève Ropars for her helpful suggestions. This research is partially supported by the Centre Henri Lebesgue (ANR-11-LABX-0020-0) and  the ANR project RAGE "Analyse Réelle et Géométrie" (ANR-18-CE40-0012).

%%  MR numbers will be added by VTeX.                      %%
%%                                                         %%
%%  Use \cite{...} to cite references in text.             %%
%%                                                         %%
%%%%%%%%%%%%%%%%%%%%%%%%%%%%%%%%%%%%%%%%%%%%%%%%%%%%%%%%%%%%%

%% if your bibliography is in bibtex format, uncomment commands:
\bibliographystyle{plain} % Style BST file (imsart-number.bst or imsart-nameyear.bst)
\bibliography{bibliography}       % Bibliography file (usually '*.bib')

\end{document}